\documentclass{article}

\usepackage[T1]{fontenc}
\usepackage{lmodern}
\usepackage{geometry}

\usepackage{amsthm,amsfonts}
\usepackage[sumlimits]{amsmath}
\usepackage{amssymb}
\usepackage[all]{xy}
\usepackage{tikz}
\usepackage{graphicx}
\usepackage{url}
\usepackage{tikz-cd}
\usepackage{pgf}
\usepackage{ragged2e}
\usepackage{hyperref}
\usepackage{tikz}
\usetikzlibrary{arrows,decorations.pathmorphing,backgrounds,positioning,fit,petri,calc}
\usepackage{mathrsfs,diagbox}
\usepackage{faktor}
\usepackage{enumerate}

\usepackage[utf8]{inputenc} 
\usepackage[T1]{fontenc}    
\usepackage{booktabs}       
\usepackage{nicefrac}       
\usepackage{microtype}      
\usepackage{lipsum}

\newtheorem{lema}{Lemma}[section]
\newtheorem{teo}[lema]{Theorem}
\newtheorem{cor}[lema]{Corollary}
\newtheorem{prop}[lema]{Proposition}
\theoremstyle{definition}

\numberwithin{equation}{section}

\newtheorem*{theoremA}{Theorem A}
\newtheorem*{theoremB}{Theorem B}
\newtheorem*{theoremC}{Theorem C}
\newtheorem*{corA}{Corollary A}

\newcommand{\cf}{\emph{cf.} }
\newcommand{\tq}{\,|\,}
\newcommand{\ie}{\emph{i.e.}}

\newcommand{\ed}{\ar@{-}}

\renewcommand{\epsilon}{\varepsilon}

\numberwithin{equation}{section}

\newcommand{\BN}{\mathbb{N}}

\newcommand{\BQ}{\mathbb{Q}}
\newcommand{\BC}{\mathbb{C}}

\newcommand{\N}{\mathbb{N}}

\newcommand{\uni}{\mathcal U}
\newcommand{\centro}{\mathcal Z}

\newcommand{\ds}{\displaystyle}

\newcommand{\inv}{^{-1}}

\newcommand{\dem}{\begin{proof}}
\newcommand{\cqd}{\end{proof}}
\newcommand{\E}{\varepsilon}
\newcommand{\che}{\left\{}
\newcommand{\chd}{\right\}}
\newcommand{\coe}{\left[}
\newcommand{\cod}{\right]}
\newcommand{\pae}{\left(}
\newcommand{\pad}{\right)}	
\newcommand{\lan}{\langle}
\newcommand{\ran}{\rangle}

\newcommand{\setad}{\rightarrow}
\newcommand{\Setad}{\Rightarrow}

\newcommand{\Setae}{\Leftarrow}
\newcommand{\fii}{\varphi}

	\newcommand{\properpagestyle}

\DeclareMathOperator{\supp}{supp}

\DeclareMathOperator{\Id}{Id}

\DeclareMathOperator{\Proj}{Proj}\DeclareMathOperator{\proj}{Proj}
\DeclareMathOperator{\ch}{char}
\DeclareMathOperator{\idemp}{Idemp}

\title{On *-Clean Group Rings over SLC-groups}

\author{
    \textbf{Kisnney Almeida} \\
    Universidade Estadual de Feira de Santana,\\
Departamento de Ciências Exatas,\\
Av. Transnordestina S/N, CEP 44036-900 - Feira de Santana - BA - Brazil. \\
    \texttt{kisnney@gmail.com}
    \and
    \textbf{Jacqueline Costa Cintra} \\
    Universidade Estadual de Feira de Santana,\\
Departamento de Ciências Exatas,\\
Av. Transnordestina S/N, CEP 44036-900 - Feira de Santana - BA - Brazil. \\
    \texttt{jccintra@uefs.br}
    \and
    \textbf{Mauricio Araujo Ferreira} \\
    Universidade Estadual de Feira de Santana,\\
Departamento de Ciências Exatas,\\
Av. Transnordestina S/N, CEP 44036-900 - Feira de Santana - BA - Brazil. \\
    \texttt{maferreira@uefs.br}
    \and
    \textbf{Edward Landi Tonucci} \\
    Universidade Estadual de Feira de Santana,\\
Departamento de Ciências Exatas,\\
Av. Transnordestina S/N, CEP 44036-900 - Feira de Santana - BA - Brazil. \\
    \texttt{eltonucci@uefs.br}
}

\begin{document}

\maketitle

\abstract{The property of $*$-cleanness in group rings has been studied for some groups considering the classical involution, given by $g^*=g^{-1}$. A group is called an SLC-group if its quotient by its center is isomorphic to the Klein group; these groups are equipped with its own canonical involution,  which usually does not coincide with the classical one. In this paper we study the $*$-cleanness of $RG$ when $G$ is an SLC-group, considering $*$ as its canonical involution. In that context, we prove that if $RG$ is $*$-clean then $G$ is the direct product of $Q_8$ and an abelian group with some extra properties and we find a converse for some specific cases, generalizing a result by Gao, Chen and Li for $Q_8$.}

\section{Introduction}

Clean rings were introduced by Nicholson in 1977 \cite{N77}, in the context of exchange rings, and have since attracted a lot of attention. A ring is said to be {\bfseries clean} if each of its elements can be written as the sum of a unit and an idempotent.

We define a {\bfseries ring involution $*$} in a ring $R$ as an antiautomorphism of order $2$, \ie, an application $*:R\to R$ such that
$$(r+s)^*=r^*+s^*,\quad (rs)^*=s^*r^*\quad \text{and}\quad (r^*)^*=r,$$
for all $r,s\in R$. In this case, we call $R$ a $*$-ring.

Va\v{s} \cite{V10} defined a {\bf $*$-clean ring} as a $*$-ring for which all their elements can be written as the sum of a unit and a projection - an idempotent $*$-invariant element. Naturally, since every $*$-clean ring is clean, the investigation on $*$-cleanness may be seen as determining conditions for a clean ring to be $*$-clean.

Similarly to the ring case, a {\bfseries group involution} in a group $G$ is a group antiautomorphism  of order 2. Given a group involution $*$ in a group $G$,
 it may be linearly extended to a ring involution for the group ring $RG$, which we also denote by $*$. An important special case is given by the classical involution $*$ defined by the inversion $g^*=g^{-1}$ in any group $G$, for which the $*$-cleanness has been widely studied \cite{GCL15, HLT15, HLY15, LPY15}.

  Assuming $*$ as the linear extension of any involution on $G$, in \cite{JM06} the authors asked when the set $(RG)^+=\che\alpha\in RG:\alpha^*=\alpha\chd$ is commutative. In that paper, it was proved that if $G$ is non-abelian and $char(R)\neq2$, then $(RG)^+$ is commutative if and only if $G'=\che1,s\chd$ and the involution $*$ on $G$ is given by
$$x^*=\left\{\begin{array}{rl}
x&\text{if}\ x\in\mathcal{Z}(G)\\
sx&\text{if}\ x\notin\mathcal{Z}(G).
\end{array}
\right.$$

The above application generally does not define an involution on $G$ and if $G$ is non-abelian this map is an involution if and only if $G'=\che 1,s\chd$ and $G$ has the {\bf lack of commutativity} property (LC for short), i.e., if $xy=yx$, then $x,y$ or $xy$ belongs to $\centro(G)$. LC-groups such that $G'=\che 1,s\chd$ are called {\bf SLC-groups} and the involution defined above is referred to be the {\bf canonical involution} (of SLC-groups). That definition is equivalent to $G/\centro(G)\simeq C_2\times C_2$. For more details, \cf \cite{GJP96,PM22}.

SLC-groups equipped with this involution appear as solutions to a series of other problems concerning group rings with involution, such as normality ($\alpha^*\alpha=\alpha\alpha^*$) \cite{CV20}, Lie identities in $(\mathbb{K} G)^+$ \cite{LSS09} and commutativity and anticommutativity in $(RG)^+$ and $(RG)^-=\che \alpha\in RG:\alpha^*=-\alpha\chd$ for linear and some non-linear extension of the involution $*$ in $G$ \cite{BP06,GP13a,GP13b,GP14,PT20}.

Due to the quantity and variety of problems in which SLC-groups with canonical involution appear as solutions in the context of group rings with involution, it is quite natural to ask when $RG$ is $*$-clean for $G$ in this class of groups. As far as we know, there are no currently published results on $*$-cleanness of those group rings, except for $G=Q_8$ \cite{GCL15}, for which the canonical involution coincides with the classical one. We will try to fill that gap in this paper.

Our main result is the following:

\begin{theoremA}
    Let $G$ be an SLC-group with canonical involution $*$ and $R$ be a unital commutative ring such that $2\in U(R)$. If $RG$ is $*$-clean then $G=Q_8\times A$, with $A$ abelian, such that

\begin{enumerate}
    \item
    The group $A$ has no elements of order 4;

    \item
    If $p$ is a prime and there is $n\in\BN$ such that $p$ divides $2^n+1$ then $A$ has no elements of order $p$.

  Besides, the equation $X^2+Y^2+Z^2+1=0$ has no solutions in $R$.
\end{enumerate}
\end{theoremA}

From the definitions, it is clear that $*$-cleanness implies cleanness, so it is a reasonable task to look for conditions that are sufficient for clean rings to be $*$-clean. With that in mind, we have generalized a result for $Q_8$ by Chen, Gao and Li \cite{GCL15}, to any direct product with finite elementary abelian 2-groups.

\begin{theoremB}
    Let $G=Q_8\times P_2$, with $P_2$ being is a finite elementary abelian 2-group, $*$ being its standard SLC-group involution and let $R$ be a unital commutative ring such that $2\in U(R)$. Then $RG$ is $*$-clean if and only if $RG$ is clean and the equation $X^2+Y^2+Z^2+1=0$ has no solutions in $R$.
\end{theoremB}

We may also restrict the ring instead of restricting the group to obtain the following.

\begin{theoremC}
    Let $R=\oplus_{i=1}^n\mathbb{F}_i$ be a semisimple ring and let $G=Q_8\times A$ with $A$ finite abelian and $*$ being its canonical SLC-group involution such that $\ch(R)$ does not divide  $|G|$. Then  $R G$ is $*$-clean if and only if the equation $X^2+Y^2+Z^2+1=0$ has no solution in $\mathbb{F}_i(\zeta_d)$, where $\zeta_d$ is a $d$-th primitive root of 1, for each $d\in\BN$ such that $A$ has an element of order $d$.
\end{theoremC}

The following corollary shows some explicit examples of $*$-clean group rings over SLC-groups.

\begin{corA}\label{corqcp}
    Let $G=Q_8\times A$, where $A$ is a finite abelian group and $*$ is its canonical involution as an SLC-group.

    \begin{enumerate}
        \item
        If $G$ contains an element of prime order $p$ such that $p\equiv 3, 5 \mod 8$ then  $\BQ G$ is not $*$-clean;

        \item
        If $A\simeq C_p$, where $p$ is a prime such that $p\equiv 7 \mod 8$ and $C_p$ is the cyclic group of order $p$, then $\BQ G$ is $*$-clean.
    \end{enumerate}
\end{corA}

This paper is organized as follows: In Section \ref{secgen} we prove some general results on $*$-rings we are going to need; in Section \ref{seclemmarings} we prove a necessary technical lemma on rings; in Section \ref{secbasic} we establish some notations and a preliminary result; in Section \ref{secnec} we prove Theorem A and in Section \ref{secsuf} we prove the other three results.

\section{ Some general results}\label{secgen}

We denote the set of units, idempotents and projections of a ring $R$ by $\uni(R)$, $\idemp(R)$ e $\proj(R)$, respectively. We will always assume $*$, $\circ$ and $\star$ are involutions in the corresponding rings.

We begin by presenting some results that are true for any rings with involution.

\begin{lema}\label{prop.involucao1ver2} Let $R$ be a unital $*$-ring, $S$ be a unital $\circ$-ring and $\fii:R\setad S$ be a homomorphism for which holds $\fii(a^*)=\fii(a)^\circ$.
\begin{enumerate}
\item If $r\in R$ is $*$-clean then $\fii(r)$ is $\circ$-clean.
\item If $\fii$ is an epimorphism and $R$ is $*$-clean then $S$ is $\circ$-clean.

\item If $\fii$ is an isomorphism then $R$ is $*$-clean if and only if $S$ is $\circ$-clean.
\end{enumerate}
\end{lema}

\dem
For the proof of item 1, if $r=u+p$, where $u\in \uni(R)$ and $p\in \proj(R)$, then $\fii(r)=\fii(u)+\fii(p)$, with $\fii(u)\in\uni(S)$ e $\fii(p)\in \proj(S)$. Item 2 follows easily from item 1 and item 3 follows from item 2.

\cqd

Cleanness interacts well with direct sum, as we can see in the following proposition.

\begin{prop}\label{limpezasomadireta.inducao}\cite[Proposition 2]{AC02} Let $R = \oplus_{i=1}^n R_i$ be a unital ring. Then $R$ is a clean ring if and only if each $R_i$ is clean.
\end{prop}

We obtain a similar result for $*$-cleanness by using the following.

\begin{prop}\cite[Proposition 4]{V10} Let $R$ be a $*$-ring. If $p$ is a projection of $R$ such $pRp$ and $(1-p)R(1-p)$ are both $*$-clean, then $R$ is $*$-clean.
\label{prop4.v10}
\end{prop}

\begin{lema}\label{slimpezasomadireta}
Let $R_1$, $R_2$ be rings such that $R=R_1\oplus R_2$ is a unital ring. If $R$ is  a $*$-ring such that   $\circ=*_{|R_1}$ and $\star=*_{|R_2}$ are involutions in $R_1$ e $R_2$ respectively, then $R$ is $*$-clean if and only if $R_1$ and $R_2$ are $\circ$-clean and $\star$-clean respectively.
\end{lema}

\dem

First note that since $R=R_1\oplus R_2$ then $1_R=e_1+e_2$, with $e_i$ being the unity of $R_i$ for $i=1,2$.

($\Rightarrow$) It follows from Lemma \ref{prop.involucao1ver2} by using the canonical projections as $\fii$.

($\Setae$) Since $R_1=e_1Re_1$, $R_2=(1-e_1)R(1-e_1)$ and $e_1\in \proj(R)$ then the result follows from Proposition \ref{prop4.v10}.
\cqd

\begin{prop}\label{slimpezasomadireta.inducao} Let $R = \oplus_{i=1}^n R_i$ be a unital ring. If $R$ has an involution $*$ such that $*_i = *_{|R_i}$ are involutions on $R_i$, respectively, then $R$ is $*$-clean if and only if each $R_i$ is $*_i$-clean.
\end{prop}

\begin{proof}
    Induction and Lemma \ref{slimpezasomadireta}.
\end{proof}

Now we are going to approach how $*$-cleanness of group rings interacts with direct product with finite elementary abelian 2-groups on the basis group.

We recall the following well-known result.

\begin{prop}\cite[Proposition 3.6.7]{PS02} Let $R$ be a unital ring and let $H$ be a normal subgroup of $G$. If $|H|\in\uni(R)$, then
$$RG\simeq R(G/H)\oplus \Delta(G,H).$$
\label{prop.3.6.7.PS02}
\end{prop}

\begin{teo}\label{extensaoc2} Let $R$ be a unital commutative ring such that $2\in\uni(R)$, let $P_2$ is a finite elementary abelian $2$-group and $G=H\times P_2$ be a group with an involution $*$ such that $*_{|P_2}=\Id_{P_2}$. If $RH$ is $*$-clean, then $RG$ is $*$-clean.
\end{teo}

\dem
Note that $G=H\times\underbrace{C_2\times\cdots\times C_2}_{k \text{ factors}}$, where $C_2$ is the cyclic group of order 2.

We will prove the result for $k=1$ and the result easily follows by induction on $k$. So assume $k=1$ hence $P_2=C_2$.

 Note that $C_2 \triangleleft G$ and $|C_2| = 2 \in \uni(R)$, thus by Proposition \ref{prop.3.6.7.PS02} we have
$$ RG \simeq R(G/C_2) \oplus \Delta(G,C_2) \simeq RH \oplus \Delta(G,C_2). $$

Let $a$ be the generator of $C_2$, and note that if $\alpha \in \Delta(G, C_2)$ then there exist $\alpha_{hi} \in R$, for $i = 0, 1$ and $h \in H$, such that
$$ \begin{array}{rcl}
\alpha &=& \displaystyle\sum_{h \in H, i=0,1} \alpha_{hi} h a^i(1 - a) \\
&=& \displaystyle\sum_{h \in H} \alpha_{h0} h(1 - a) + \sum_{h \in H} \alpha_{h1} h a(1 - a) \\
&=& \displaystyle\sum_{h \in H} \alpha_{h0} h(1 - a) + \sum_{h \in H} \alpha_{h1} h(a - 1) \\
&=& \displaystyle\sum_{h \in H} \alpha_{h0} h(1 - a) - \sum_{h \in H} \alpha_{h1} h(1 - a) \\
&=& \displaystyle\sum_{h \in H} (\alpha_{h0} - \alpha_{h1}) h(1 - a),
\end{array} $$
from which we can assume, for simplicity, that $\alpha = \dfrac{r}{2}(1 - a)$ with $r=2\sum_{h\in H}(\alpha_{h0}-\alpha_{h1})h \in RH$.

Note that $1_{\Delta(G, C_2)} = \dfrac{1 - a}{2}$, since $(1 - a)^2 = 2(1 - a)$, and thus $$r(1 - a) \cdot \dfrac{1 - a}{2} = r \cdot \dfrac{2(1 - a)}{2} = r(1 - a).$$

We will now show that $\Delta(G, C_2)$ is $*$-clean.

Let $\alpha = \dfrac{r}{2}(1 - a) \in \Delta(G, C_2)$, and we will show that $\alpha$ is a $*$-clean element of $\Delta(G, C_2)$.

In fact, from the $*$-cleanness of $RH$ we know that if $r \in RH$ then $r = u + p$, with $u \in \uni(RH)$ and $p^2 = p = p^*$. Thus, $\alpha = \dfrac{r}{2}(1 - a) = \left(\dfrac{u}{2} + \dfrac{p}{2}\right)(1 - a)$.

We will show that $\dfrac{u}{2}(1 - a) \in \uni(\Delta(G, C_2))$ and that $\left(\dfrac{p}{2}(1 - a)\right)^2 = \dfrac{p}{2}(1 - a) = \left(\dfrac{p}{2}(1 - a)\right)^*$.

To show that $\dfrac{u}{2}(1 - a) \in \uni(\Delta(G, C_2))$, it suffices to note that
$$ \dfrac{u}{2}(1 - a) \cdot \dfrac{u^{-1}}{2}(1 - a) = \dfrac{uu^{-1}}{4}(1 - a)^2 = \dfrac{1 - a}{2}. $$

Finally, note that
$$ \left(\dfrac{p}{2}(1 - a)\right)^2 = \dfrac{p^2}{4}(1 - a)^2 = \dfrac{p}{4} \cdot 2(1 - a) = \dfrac{p}{2}(1 - a) $$
and that $\dfrac{p}{2}(1 - a) = \left(\dfrac{p}{2}(1 - a)\right)^*$ follows from the fact that $p^* = p$ and $(1 - a)^* = (1 - a)$, since $a \in \centro(G)$.

Thus we have that $\Delta(G, C_2)$ is $*$-clean and by Lemma \ref{slimpezasomadireta} we conclude $RG$ is $*$-clean.
\cqd

\begin{cor}\label{corp2} Let $R$ be a unital ring such that $2\in\uni(R)$, let $G$ be a group and let $*$ be the classical involution. If $RG$ is $*$-clean, then $R(G\times P_2)$ is $*$-clean for any finite elementary abelian $2$-group $P_2$.
\end{cor}

\dem It suffices to note that if $*$ is the classical involution, then $*_{|P_2}=Id_{P_2}$, therefore the result follows from Theorem \ref{extensaoc2}.
\cqd

\section{A technical lemma on rings} \label{seclemmarings}

To prove Theorem A, we will need the following technical lemma, inspired by \cite[p. 77, Exercise 21]{L05}.

\begin{lema}\label{propexmau}
Let $S$ be a unital ring and $p\geq 3$ be an odd natural number. If $g\in S$ is a central element of $S$ such that $g^p=1$ then
\begin{enumerate}
\item
$$\prod _{k=0}^{t} \left(1 + g^{2^k}\right)$$ is a sum of two squares for each $t\geq 0$.

\item
Let $n\in\BN$ such that $p$ divides $2^n+1$. Then there are $\alpha,\beta\in S$ such that
$$\left(\alpha ^2+ \beta ^2+g^{2^n}\right)(g-1)=0.$$
\end{enumerate}
\end{lema}

\begin{proof}
We prove item 1 by induction on $t$. For $t=0$, note that
$$1+g=1^2+\left(g^{\frac{p+1}{2}}\right)^2.$$

If the result is true for $t\geq 0$, then there are $a,b\in S$ such that
$$\prod_{k=0}^{t+1} \left(1 + g^{2^k}\right)=\left(a^2+b^2\right)\left(1+g^{2^{t+1}}\right)=\left(ag^{2^t}+b\right)^2 + \left(bg^{2^t}-a\right)^2,$$
which proves item 1.

For the proof of item 2, note that by hypothesis  $2^n-1\equiv p-2\mod p$, so there is a natural $q$ such that  $2^n-1=pq+p-2$ hence

\begin{multline*}
\Pi:= \prod_{k=0}^{n-1} \left(1 + g^{2^k}\right)=\sum_{i=0}^{2^n-1}g^i=\left(\sum_{j=0}^{q-1} g^{pj}\sum_{i=0}^{p-1} g^i\right)+ g^{pq}\sum_{i=0}^{p-2}g^i=q\sum_{i=0}^{p-1} g^i + \sum_{i=0}^{p-2}g^i=\\=\left((q+1)\sum_{i=0}^{p-1} g^i\right) - g^{p-1}.
\end{multline*}

Since $2^n\equiv p-1\mod p$ then $g^{2^n}=g^{p-1}$ hence
$$\left(\Pi+g^{2^n}\right)(g-1)=\left(\Pi+g^{p-1}\right)(g-1)=\left[(q+1)\sum_{i=0}^{p-1} g^i\right](g-1)=(q+1)(g^p-1)=0 $$
and the result follows from item 1.
\end{proof}

\section{SLC-groups and Group Rings}\label{secbasic}

We recall a group $G$ is an SLC-group if and only if $\faktor{G}{\mathcal{Z}(G)}$ is isomorphic to the Klein group. It easily follows that the SLC property is closed for direct product with abelian groups. With that in mind, the following result gives a full description of SLC-groups in terms of presentations which we will freely use.

\begin{teo}[\cite{JLPM95}]\label{indecomp}
A group $G$ is a SLC-group if and only if $G=D\times A$, where $A$ is abelian and $D$ is an  indecomposable $2$-group such that $D=\langle x,y,\centro(D)\rangle$, where $D$ admits one of the following presentations:

\begin{enumerate}
\item $D_1=\lan x,y,a: x^2=y^2=a^{m}=1\ran$;
\item $D_2=\lan x,y,a: x^2=y^2=a,\ a^{m}=1\ran$;
\item $D_3=\lan x,y,a,b: x^2=a^m=b^{m_2}=1,\ y^2=b\ran$;
\item $D_4=\lan x,y,a,b: x^2=a,\ a^m=b^{m_2}=1,\ y^2=b\ran$;
\item $D_5=\lan x,y,a,b,c: x^2=b,\ y^2=c,\ a^m=b^{m_2}=c^{m_3}=1\ran$;
\end{enumerate}
such that $G'=\che 1,s\chd$ with $s=(x,y)=a^{m/2}$, $m=2^k$ and $m_i=2^{k_i}$, being $k,k_i>0$ for all $i$. For simplicity, we omit in the above presentations the relations $[\{x,y,a,b,c\}, \{a,b,c\}]=1$,  which obviously imply   $a,b,c\in\centro(G)$, whenever possible. More precisely, $\centro(G)=\lan a\ran\times K$ for some abelian group $K$, which is given below.

\begin{enumerate}
\item $K=A$ if $D=D_1,D_2$;
\item $K=\lan b\ran \times A$ if $D=D_3,D_4$;
\item $K=\lan b\ran\times \lan c\ran \times A$ if $D=D_5$.
\end{enumerate}
\end{teo}

If an SLC-group $G$ is isomorphic to $D_i\times A$, as in Theorem \ref{indecomp}, we will say $G$ is {\bfseries of type i} and we are going to freely use the notations of Theorem \ref{indecomp} for $x,y,a,b,c$.

In this section and the next, we will assume $G$ is an SLC-group with $G'=\che1,s\chd$, equipped with the canonical involution $*$, and $\tau=\che 1=t_1,t_2,t_3,t_4\chd$ is a transversal of $G$ over $\centro(G)$, typically $\{1, x, y, xy\}$.

We will also assume $R$ is a commutative unital ring such that $2\in \mathcal{U}(R)$. Hence $RG$ is a group ring such that $e=\dfrac{1+s}{2}$, $f=\dfrac{1-s}{2}$ are orthogonal idempotents of $RG$, which means $e^2=e$, $f^2=f$, $ef=fe=0$. Hence $RG=(RG)e\oplus (RG)f$. We also remark $e$ and $f$ are central projections, \ie, projections that belong to the center of $RG$.

We begin with a technical lemma which is the basis for most of our proofs.

\begin{lema}\label{lemad1d2}
 Let $R$ be a ring and $G$ an SLC-group, with the notations and assumptions above. Then

\begin{enumerate}
  \item
  Every element $\alpha\in (RG)f$ may be uniquely written as
$$\alpha=\left[\sum_{j=1}^4\pae\sum_{i=0}^{\frac{m}{2}-1}x_{ij}a^i\pad t_j\right](1-s), \quad \text{with } x_{ij}\in RK \text{ for all } i,j.$$

\item
Given $\alpha$ as above,
$$\alpha^*=\left[\sum_{i=0}^{\frac{m}{2}-1}x_{i1}a^i-\sum_{j=2}^4\pae\sum_{i=0}^{\frac{m}{2}-1}x_{ij}a^i\pad t_j\right](1-s);$$

\item
Given $\alpha$ as above,
$$\alpha=\alpha^* \Leftrightarrow x_{ij}=0\ \text{ for all } j\geq2;$$

\item
$\proj((RG)f)=\che d(1-s)\,|\,\ds d=\sum_{i=0}^{\frac{m}{2}-1}x_{i1}a^i\ \text{and}\ d=2d^2\chd$.
\end{enumerate}
\end{lema}

\begin{proof}
\begin{enumerate}

\item
First we prove the existence. Since $K$
is a transversal for $\langle a \rangle$ in $\mathcal{Z}(G)$, then

    \begin{equation}\label{transversal2}
    K\tau=\{k t_j \tq k\in K   , 1\leq j \leq 4\} \text{ is a transversal for  $\langle a \rangle$ in $G$}.
    \end{equation}

    Let $\alpha\in (RG)f$. Then there is $n\in \BN$ and $k_1,k_2,\dots,k_n\in K$ such that for each $1\leq j\leq 4$, $0\leq i\leq \frac{m}{2}-1$, $1\leq \E\leq n$,
		$0\leq\delta \leq 1$	
		there is a $\gamma_{ij\E\delta}\in R$ such that

 \begin{align*}
 \alpha&=\left(\sum_{j,i,\E, \delta} \gamma_{ji\E\delta}k_\E t_ja^is^{\delta}\right)(1-s)\\
 		&=\left[ \left(\sum_{j,i,\E} \gamma_{ji\E 0}k_\E t_ja^i\right) + \left(\sum_{j,i,\E}\gamma_{ji\E 1}k_\E t_ja^is\right) \right](1-s)\\
		&=\left[ \left(\sum_{j,i,\E} \gamma_{ji\E 0}k_\E t_ja^i\right) - \left(\sum_{j,i,\E}\gamma_{ji\E 1}k_\E t_ja^i\right) \right](1-s)\\
		&=\left(\sum_{j,i,\E} \left(\gamma_{ji\E 0}-\gamma_{ji\E 1}\right)k_\E t_ja^i\right)(1-s)\\
		&=\coe\sum_{j=1}^4\left(\sum_{i,\E} \left(\gamma_{ji\E 0}-\gamma_{ji\E 1}\right)k_\E a^i\right)t_j\cod(1-s)\\
		&=\coe\sum_{j=1}^4\left(\sum_{i=0}^{\frac{m}{2}-1}\left(\sum_{\E=1}^n \left(\gamma_{ji\E 0}-\gamma_{ji\E 1}\right)k_\E \right)a^i\right)t_j\cod(1-s)\\
		&=\coe\sum_{j=1}^4\left(\sum_{i=0}^{\frac{m}{2}-1}x_{ij}a^i \right)t_j\cod(1-s),\\
		 		,
 \end{align*}

which proves the existence.

To prove uniqueness, note that
$$\alpha=\left[\sum_{j=1}^4\pae\sum_{i=0}^{\frac{m}{2}-1}x_{ij}a^i\pad t_j\right](1-s)=\left[\sum_{j=1}^4\pae\sum_{i=0}^{\frac{m}{2}-1}x_{ij}'a^i\pad t_j\right](1-s)$$

implies

\begin{multline*}
\left[\sum_{j=1}^4\pae\sum_{i=0}^{\frac{m}{2}-1}x_{ij}a^i\pad t_j\right]-\left[\sum_{j=1}^4\pae\sum_{i=0}^{\frac{m}{2}-1}x_{ij}a^i\pad t_j\right]s =\\
=\left[\sum_{j=1}^4\pae\sum_{i=0}^{\frac{m}{2}-1}x_{ij}'a^i\pad t_j\right]-\left[\sum_{j=1}^4\pae\sum_{i=0}^{\frac{m}{2}-1}x_{ij}'a^i\pad t_j\right]s.
\end{multline*}

Since $\che t_ja^is^\E:1\leq j\leq4, 0\leq i\leq \frac{m}{2}-1,\E=0,1\chd$ is a transversal for $K$ in $G$, we obtain $x_{ij}=x'_{ij}$ for all $i,j$.

\item

Note that  $\supp(x_{ij}a^i(1-s))$ is central hence invariant by $*$ for all $i,j$. Then

\begin{align*}
\alpha^*&= \left(\left[\sum_{j=1}^4\pae\sum_{i=0}^{\frac{m}{2}-1}x_{ij}a^i\pad t_j\right](1-s)\right)^*\\
&= \pae\left[\sum_{i=0}^{\frac{m}{2}-1}x_{i1}a^i +\sum_{j=2}^4\pae\sum_{i=0}^{\frac{m}{2}-1}x_{ij}a^i\pad t_j\right](1-s)\pad^*\\
&= \pae\coe\sum_{i=0}^{\frac{m}{2}-1}x_{i1}a^i\cod(1-s)\pad^* +\pae\coe\sum_{j=2}^4\pae\sum_{i=0}^{\frac{m}{2}-1}x_{ij}a^i\pad t_j\cod(1-s)\pad^*\\
&=\coe\sum_{i=0}^{\frac{m}{2}-1}x_{i1}a^i\cod(1-s) +\coe\sum_{j=2}^4\pae\sum_{i=0}^{\frac{m}{2}-1}x_{ij}a^i\pad t_j^*\cod(1-s)\\
&=\coe\sum_{i=0}^{\frac{m}{2}-1}x_{i1}a^i\cod(1-s) +\coe\sum_{j=2}^4\pae\sum_{i=0}^{\frac{m}{2}-1}x_{ij}a^i\pad t_js\cod(1-s)\\
&=\coe\sum_{i=0}^{\frac{m}{2}-1}x_{i1}a^i\cod(1-s) +\coe\sum_{j=2}^4\pae\sum_{i=0}^{\frac{m}{2}-1}x_{ij}a^i\pad (-t_j)\cod(1-s)\\
&=\coe\sum_{i=0}^{\frac{m}{2}-1}x_{i1}a^i\cod(1-s) -\coe\sum_{j=2}^4\pae\sum_{i=0}^{\frac{m}{2}-1}x_{ij}a^i\pad t_j\cod(1-s),\\
&=\left[\sum_{i=0}^{\frac{m}{2}-1}x_{i1}a^i-\sum_{j=2}^4\pae\sum_{i=0}^{\frac{m}{2}-1}x_{ij}a^i\pad t_j\right](1-s).
\end{align*}

\item
It follows from the two previous items.

\item
Let $\alpha\in \Proj((RG)f)$. Then $\alpha=\alpha^*$ implies $\ds\alpha=\sum_{i=0}^{\frac{m}{2}-1}x_{i1}a^i(1-s)=d(1-s)$. Then
$$d(1-s)=\alpha=\alpha^2=d^2(1-s)^2=d^22(1-s).$$
By the uniqueness of item 1, $d=2d^2$. The converse follows from the centrality of $\supp (d(1-s))$.

\end{enumerate}

\end{proof}

\section{Necessary conditions}\label{secnec}

In this section we prove Theorem A. Our main tool will be the following lemma.

\begin{lema}\label{lemagamma}
  Let $R$ be a ring and $G$ be an SLC-group, with the assumptions of Section \ref{secbasic}. Suppose there are elements $\gamma,\tau\in RG$ such that:

  \begin{enumerate}
      \item
      $(1-\gamma^2)\tau(1-s)=0$;

      \item
      $(z-4\inv \gamma)\tau(1-s)\neq 0$ for all $z\in RZ(G)$.
  \end{enumerate}

  Then $RG$ is not $*$-clean.
\end{lema}

\begin{proof}
Suppose $RG$ is $*$-clean. By Lemma \ref{slimpezasomadireta}, $(RG)f$ is also $*$-clean.

Let $\gamma,\tau \in RG$ be as in the statement. Then  $(1-\gamma^2)\tau(1-s)=0$.

Consider $h=4^{-1}(1+\gamma)(1-s)\in (RG)f$. By hypothesis, there is a unit $u\in (RG)f$ and a projection $p\in (RG)f$ such that $h=u+p$.

By Lemma \ref{lemad1d2}, item 4, there is $d\in RZ(G)$ such that $p= d(1-s)$, with  $2d^2(1-s)=d(1-s)$.

Hence
$$u=h-p=\left(\left(4^{-1}-d\right) + 4^{-1}\gamma \right) (1-s).$$

Let $$v:=\left(\left(4^{-1}-d\right) - 4^{-1}\gamma \right)\tau (1-s). $$

By hypothesis, $v\neq 0$ hence

\begin{align*}
uv&=2\left(\left(16\inv -2\inv d + d^2\right) - 16\inv \gamma^2)\right)\tau(1-s)\\
&=8\inv (1-\gamma^2)\tau(1-s)=0,
\end{align*}
which is a contradiction, since $u$ is a unit and $v\neq0$.

It follows that $(RG)f$ is not $*$-clean, which implies $RG$ isn't either, by Lemma \ref{slimpezasomadireta}.
\end{proof}

Now we are ready to prove the first part of Theorem A.

\begin{lema}
Let $G$ be an SLC-group and $R$ a ring with the assumptions of Section \ref{secbasic}. If $RG$ is $*$-clean then $G=Q_8\times A$, with $A$ abelian, such that

\begin{enumerate}
    \item
    The group $A$ has no elements of order 4.

    \item
    If $p$ is a prime and there is $n\in\BN$ such that $p$ divides $2^n+1$ then $A$ has no elements of order $p$.
\end{enumerate}
\end{lema}

\begin{proof}
   We will use Lemma \ref{lemagamma} several times to eliminate possibilities, defining different  $\gamma$ and  $\tau$ for each case. In every case we use the transversal $\{t_1,t_2,t_3,t_4\}$ of $G$ over $Z(G)$  as being $\{1,x,y,xy\}$ (\cf Theorem \ref{indecomp}).

By Theorem \ref{indecomp}, $G$ is isomorphic to  $D_{i_0}\times A$, for some $i_0\in\{1,2,3,4,5\}$.

Suppose $i_0=1$. Take $\gamma=y$ and $\tau=1$. Then
\begin{align*}
(1-\gamma^2)\tau(1-s)&=(1-y^2)(1-s)=(1-1)(1-s)=0.
\end{align*}

By Lemma \ref{lemagamma}, there is $z\in RZ(G)$ such that
\begin{align*}
    0&=(z-4\inv \gamma)\tau(1-s)=(z-4\inv y)(1-s).
\end{align*}

It follows from the uniqueness part of Lemma \ref{lemad1d2} that $-4\inv=0$, a contradiction.

Suppose now $i_0\in\{3,4,5\}$. Take $\gamma=y$ and $\tau=\sum_{i=0}^{\frac{r}{2} - 1}y^{2i}$, where $r\geq 4$ is the order of $y$.

Then

\begin{align*}
(1-\gamma^2)\tau(1-s)&=(1-y^2)\left(\sum_{i=0}^{\frac{r}{2} - 1}y^{2i}\right)(1-s)\\
&=\left(\sum_{i=0}^{\frac{r}{2} - 1}y^{2i}-\sum_{i=1}^{\frac{r}{2}}y^{2i}\right)(1-s)\\
&=(1-y^r)(1-s)=(1-1)(1-s)=0.
\end{align*}

By Lemma \ref{lemagamma} there is $z\in RZ(G)$ such that
\begin{align*}
    0&=(z-4\inv \gamma)\tau(1-s)\\
    &=(z-4\inv y)\left(\sum_{i=0}^{\frac{r}{2} - 1}y^{2i}\right)(1-s)\\
    &=\left[z\left(\sum_{i=0}^{\frac{r}{2} - 1}y^{2i}\right)      + y\left(-4\inv\sum_{i=0}^{\frac{r}{2} - 1}y^{2i}\right)\right](1-s).
\end{align*}

Note that $z\sum_{i=0}^{\frac{r}{2} - 1}y^{2i}\in RZ(G)$ and, by construction, the powers $y^{2i}\neq 1$, for $1\leq i\leq \frac{r}{2} - 1$, belong to $\langle b \rangle$ or $\langle c \rangle$, never to $\langle a \rangle$. Then $4\inv\sum_{i=0}^{\frac{r}{2} - 1}y^{2i}$ belongs to $RK$. It follows that, by the uniqueness part of Lemma \ref{lemad1d2},
$$-4\inv\sum_{i=0}^{\frac{r}{2} - 1}y^{2i}=0,$$
which is a contradiction, since each power $y^{2i}$, for $0\leq i \leq \frac{r}{2} - 1$, is a different element of the group $G$.

That means $i_0=2$. Suppose $m\geq 4$. Take $\gamma=xya^{\frac{m-4}{4}}$ e $\tau=1$. We have
\begin{align*}
(1-\gamma^2)\tau(1-s)&=\left(1-\left(xya^{\frac{m-4}{4}}\right)^2\right)(1-s)\\
&= \left(1-x^2y^2sa^{\frac{m-4}{2}}\right)(1-s)\\
&= \left(1-a^2sa^{\frac{m-4}{2}}\right)(1-s)\\
&= \left(1-sa^{\frac{m}{2}}\right)(1-s)\\
&= \left(1-s^2\right)(1-s)=(1-1)(1-s)=0.
\end{align*}

By Lemma \ref{lemagamma}, there is $z\in RZ(G)$ such that
$$
    0=\left(z-4\inv \gamma\right)\tau(1-s)=\left(z+ xy\left(-4\inv a^{\frac{m-4}{4}}\right) \right) (1-s).
$$

By the uniqueness part of Lemma \ref{lemad1d2}, we have $-4\inv a^{\frac{m-4}{4}}  = 0$, a contradiction.

Then $i_0=2$ and $m=2$, which means $G\simeq Q_8\times A$, for some abelian group $A$.

For the proof of item 1, suppose $A$ has an element of order 4, which we denote by $g\in A$. Take $\gamma=xg$ and $\tau=1-g^2$. Then
\begin{align*}
(1-\gamma^2)\tau(1-s)&=(1-(xg)^2)(1-g^2)(1-s)\\
&=(1-x^2g^2)(1-g^2)(1-s)\\
&=(1+g^2)(1-g^2)(1-s)\\
&=(1-g^4)(1-s)=0(1-s)=0.
\end{align*}

By Lemma \ref{lemagamma}, there is $z\in RZ(G)$ such that
\begin{align*}
    0&=(z-4\inv \gamma)\tau(1-s)\\
    &=(z-4\inv xg)(1-g^2)(1-s)\\
    &=\left[z(1-g^2) + x \left( -4\inv (1-g^2)\right)       \right](1-s).
\end{align*}

By the uniqueness part of Lemma \ref{lemad1d2}, we have $-4\inv (1-g^2)= 0$, which is a contradiction, since $g$ has order 4.

Finally, for the proof of item 2, suppose  $g\in A$ has order $p$ as in the statement.  Since $S=RA,p,g$ satisfy the hypothesis of Lemma \ref{propexmau}, then there are $\alpha,\beta\in RA\leq Z(RG)$ such that

\begin{equation}\label{eqp2n}
\left(\alpha^2+\beta^2+g^{2^n}\right)(g-1)=0.
\end{equation}

Take $\gamma=\alpha g^{2^{n-1}+1}x+\beta g^{2^{n-1}+1}y$ and $\tau=g-1$.

Then
\begin{align*}
(1-\gamma^2)\tau(1-s)&=\left[ 1 -  \left(\alpha^2 g^{2^n+2}x^2+\alpha\beta g^{2^n+2}xy+\alpha\beta g^{2^n+2}yx+\beta^2 g^{2^n+2}y^2 \right) \right](g-1)(1-s)\\
&=\left[ 1+g(\alpha^2+\beta^2)\right](g-1)(1-s)\\
&=\left[ g^{2^n+1}+g(\alpha^2+\beta^2)\right](g-1)(1-s)\\
&=g\left[g^{2^n}+\alpha^2+\beta^2\right](g-1)(1-s)=0.
\end{align*}

By Lemma \ref{lemagamma}, there is $z\in RZ(G)$ such that
\begin{align*}
    0&=(z-4\inv \gamma)\tau(1-s)\\
    &=\left[z-4\inv  \left( \alpha x g^{2^{n-1}+1} + \beta y g^{2^{n-1}+1} \right)  \right](g-1)(1-s)\\
    &=\left[z(g-1) - 4\inv \alpha  g^{2^{n-1}+1} (g-1)x - 4\inv \beta  g^{2^{n-1}+1}(g-1)y  \right](1-s).
\end{align*}

The uniqueness part of Lemma \ref{lemad1d2} implies

$$4\inv \alpha  g^{2^{n-1}+1} (g-1)=4\inv \beta  g^{2^{n-1}+1}(g-1)=0$$
hence $\alpha(g-1)=\beta(g-1)=0$.

Combining the above with \eqref{eqp2n}, we obtain $1-g^{2^n}=g^{2^n}(g-1)=0$, which is a contradiction since $2^n\equiv -1 \mod p$.
\end{proof}

Then Theorem A follows from the next Lemma, inspired by \cite[Theorem 3.8]{GCL15}.

\begin{lema} Let $R$ be a ring and $G=Q_8\times A$, with $A$ being an abelian group, and with the assumptions of Section \ref{secbasic}. If the equation $X^2+Y^2+Z^2=1$ has a solution in $R$ then $RG$ is not $*$-clean.
\end{lema}

\begin{proof}
    Let $a_1,a_2,a_3\in R$ be such that
    $$a_1^2+a_2^2+a_3^2+1=0.$$
Note that $a_i$ may not be all equal to zero.
Suppose $RG$ is $*$-clean. We are going to use Lemma \ref{lemagamma} to find a contradiction. We recall $G$ is of type 2 with $D_2=Q_8$, hence $x^2=y^2=(xy)^2=s$ and $xy=yxs$. Let $$\gamma:=2\inv (a_1x+a_2y+a_3xy)(1-s)\in RG$$
and $\tau=1\in RG$.

Then
\begin{align*}
    \gamma^2&=4\inv \left((a_1x+a_2y+a_3xy)^2\right)(1-s)^2\\
    &=4\inv \left[(a_1^2+a_2^2+a_3^2)s\right]2(1-s)\\
    &=4\inv \cdot 2[(-1)s](1-s)=2\inv (1-s).
\end{align*}

Hence

$$\left(1-\gamma^2\right)\tau(1-s)=2\inv (1+s)(1-s)=0.$$

By Lemma \ref{lemagamma} there is $z\in RZ(G)$ such that
$$0=(z-4\inv \gamma)\tau(1-s)=\left(z-4\inv  a_1x -4\inv a_2y -4\inv a_3xy\right)(1-s).$$

Now the uniqueness part of Lemma \ref{lemad1d2} implies $a_1=a_2=a_3=0$, which is a contradiction.

\end{proof}

We finish this section with a simple corollary.

\begin{cor}
  If $G$ is an SLC-group with canonical involution $*$, $\BC G$ is not $*$-clean.
\end{cor}

\begin{proof}
    Since the equation $X^2+Y^2+Z^2+1=0$ has a solution in $\BC$, namely $(X,Y,Z)=(i,0,0)$, then the result follows from Theorem A.
\end{proof}

\section{Sufficient conditions}\label{secsuf}

In this section we investigate sufficient conditions for clean group rings in the pattern of Theorem A to be $*$-clean. Before our results, we present the theorem that inspired this work.

\begin{teo}\label{teocgl}\cite[Theorem 3.8]{GCL15}
    Let $R$ be a unital commutative ring such that $2\in\uni(R)$  and let $*$ be the canonical involution on the SLC-group $Q_8$ (which coincides with the inversion). Then

    \begin{enumerate}
        \item
        If $RQ_8$ is $*$-clean then $RQ_8$ is clean and the equation $X^2+Y^2+Z^2+1=0$ has no solutions in $R$.

        \item
        If $R$ is local,  $RQ_8$ is clean and the equation $X^2+Y^2+Z^2+1=0$ has no solutions in $R$, then $RQ_8$ is $*$-clean.
    \end{enumerate}
\end{teo}

We remark that in the original version of the above theorem $R$ is local in both itens, but the proof of the first item does not use locality. Our goal in this section is to generalize Theorem \ref{teocgl} as much as possible to other SLC-groups, possibly adding hypothesis on $R$.

If $A$ is an elementary abelian 2-group, the result follows easily from Section \ref{secgen}, as we can see below. It is worth noting that in this case the canonical SLC involution coincides with the classical one.

\begin{proof}[Proof of Theorem B]
If $RG$ is $*$-clean, it is obviously clean and the equation $X^2+Y^2+Z^2+1=0$ has no solution in $R$ by Theorem A.

Suppose $RG$ is clean and the above equation has no solution in $R$. By definition we have $*_{|P_2}=Id _{P_2}$. Hence the result follows from Theorems \ref{teocgl} and \ref{extensaoc2} for $H=Q_8$.
\end{proof}

The converse of Theorem \ref{teocgl}  may also be easily extended to applying direct sums on the coefficients.

\begin{teo} Let $R$ be a unital commutative ring such that $2\in\uni(R)$ and $R=\oplus_{i=1}^kR_i$ with $R_i$ local for all $i$. If $RQ_8$ is clean and the equation $x^2+y^2+z^2+1=0$ has no solution in any $R_j$ then $RQ_8$ is $*$-clean.
\label{teo.ri.local}
\end{teo}

\begin{proof}
Note that $RQ_8=(\oplus_{i=1}^kR_i)Q_8\simeq \oplus_{i=1}^k(R_iQ_8)$. Then by Proposition \ref{limpezasomadireta.inducao} $R_iQ_8$ is clean hence $*$-clean for all $i$ by Theorem \ref{teocgl}. By applying Lemma \ref{slimpezasomadireta}, we obtain $RQ_8$ is $*$-clean.
\end{proof}

We may also easily obtain a first generalization for semisimple coefficients rings.

\begin{teo} Let $R$ be a unital commutative ring such that $2\in\uni(R)$ and $G=Q_8\times A$, with $A$ being a finite abelian group. If $R$ is semisimple, $|G|\in\uni(R)$, $RG$ is clean and the equation $X^2+Y^2+Z^2+1=0$ has no solution in any of the simple components of $RA$ then $RG$ is $*$-clean.
\label{teo.ra.semissimples}
\end{teo}

\begin{proof}
 First note that $|G|\in\uni(R)$ implies $|A|\in\uni(R)$ hence, since $R$ is semisimple, by Maschke's Theorem we conclude $RA$ is semisimple. That means $RA=\oplus_{i=1}^k\mathbb{F}_i$, with $\mathbb{F}_i$ begin a field for all $i$. Since every field is local,
$$RG\simeq (RA)Q_8\simeq (\oplus_{i=1}^k\mathbb{F}_i)Q_8$$
and by hypothesis the equation $X^2+Y^2+Z^2+1$ has no solution in any of the fields $\mathbb{F}_i$, Theorem \ref{teo.ri.local} implies $RG$ is $*$-clean.

\end{proof}

 For a deeper investigation we need the following result.

\begin{teo}\cite[Corollary 9]{Z10}\label{teoz}
    If $R$ is a commutative local ring such that $J(R)$ is nil and $G$ is a locally finite group, then $RG$ is clean.
\end{teo}

\begin{teo}\label{teofields} Let $\mathbb{F}$ be a field and Let $G=Q_8\times A$, where $A$ is a finite abelian group and $*$ is its canonical involution as an SLC-group. Suppose also that  $\ch(\mathbb{F})$ does not divide $|G|$. Then  $\mathbb{F} G$ is $*$-clean if and only if the equation $X^2+Y^2+Z^2+1=0$ has no solution in $\mathbb{F}(\zeta_d)$, where $\zeta_d$ is a $d$-th primitive root of 1, for each  $d\in \BN$ such that $A$ has an element of order $d$.
\label{reciproca.corpos}
\end{teo}

\dem

First note that by Perlis-Walker Theorem (\cf \cite[Theorem 3.5.4]{PS02}])  we have
$$\mathbb{F} A\simeq \mathop{\oplus}\limits_{d} a_d \mathbb{F} (\zeta_d),$$
where $\zeta_d$ is a $d$-th primitive root of 1, with $n=|A|$ and $a_d\in \N$ for each $d\in\BN$ such that $A$ has an element of order $d$.

($\Setad$) Suppose $\mathbb{F}G$ is $*$-clean. Since $G$ is  finite and $J(\mathbb{F})=\che0\chd$, Theorem \ref{teoz} implies $\mathbb{F}G$ is clean.

Since $\mathbb{F} G\simeq \mathbb{F}(Q_8\times A)\simeq (\mathbb{F} A)Q_8$ and in $Q_8$ the classical involution coincide with the standard involution of $Q_8$, by Lemma \ref{prop.involucao1ver2} we have $(\mathbb{F} A)Q_8$ is $\circ$-clean, where $\circ$ is the classical involution of  $Q_8$.

So
$$(\mathbb{F} A)Q_8\simeq \oplus_{d|n}a_d(\mathbb{F} (\zeta_d)Q_8)$$
and by Proposition \ref{slimpezasomadireta.inducao}, each $\mathbb{F} (\zeta_d)Q_8$ is $\circ$-clean.

Applying Theorem \ref{teocgl} to each $\mathbb{F} (\zeta_d)Q_8$, we have that the equation $X^2+Y^2+Z^2+1=0$ has no solution in  $\mathbb{F} (\zeta_d)$.

($\Setae$) Since the simple components of $\mathbb{F} A$ are the fields  $\mathbb{F} (\zeta_d)$,
by applying Theorem \ref{teo.ra.semissimples} we conclude $\mathbb{F} G$ is $*$-clean.
\cqd

Now we can prove Theorem C.

\begin{proof}[Proof of Theorem C]
First note that $RG\simeq \oplus_{i=1}^n(\mathbb{F}_i G)$ and the restriction $*_i$ of $*$ to $\mathbb{F}_i G$ is an involution of $\mathbb{F}_i G$. By Proposition \ref{slimpezasomadireta.inducao} $RG$ is $*$-clean if and only if each $\mathbb{F}_i G$ is $*_i$-clean. The result follows from applying Theorem \ref{teofields} to each $\mathbb{F}_i G$.
\end{proof}

If we consider rational coefficients, we may use quadratic forms theory to achieve some extra conclusions. We recall that the {\it level} $s(\mathbb{F})$ of a field $\mathbb{F}$ is the smallest natural number $n$ such that $-1$ is a sum of $n$ squares in $F$. If $-1$ is not a sum of squares, then $s(\mathbb{F}) = \infty$. By the Pfister Level Theorem, $s(\mathbb{F})$ is always $\infty$ or a power of 2 (\cf \cite[p. 379]{L05}). For the specific case of  $\mathbb{Q}(\zeta_p)$, where $p$ is an odd prime and $\zeta_p$ is a $p$-th primitive root of 1, we have

$$
    s(\mathbb{Q}(\zeta_p))=\left\{ \begin{array}{cl}
    2, &\text{ if }p\equiv 3,5 \mod 8\\
    4, &\text{ if }p\equiv 7 \mod 8\\
    2 \text{ or } 4, &\text{ if } p\equiv 1 \mod 8.
    \end{array}\right.
$$

\begin{proof}[Proof of Corollary A]
\begin{enumerate}
    \item
    Since  $s(\mathbb{Q}(\zeta_p))=2$ then the equation $X^2+Y^2=-1$ has a solution in $\mathbb{Q}(\zeta_p)$ hence, by making $Z=0$, $\BQ G$ is not $*$-clean by Theorem A.

    \item
    Since $s(\mathbb{Q}(\zeta_p))=4$, then the equation $X^2+Y^2+Z^2+1=0$ has no solutions in $\mathbb{Q}(\zeta_p)$. Since every element of $C_p$ has order $p$, the result follows from Theorem \ref{teofields}.
\end{enumerate}
\end{proof}

We finish with an interesting application of Corollary A. The following result may be proven by using the law of quadratic reciprocity, but we give a different proof.

\begin{cor}
    Let $p$ be a prime number. If $p\equiv 7 \mod 8$ then there is no $n\in \BN$ such that $p$ divides $2^{n}+1$.
\end{cor}

\begin{proof}
    Suppose there is such $n$ and let $G=A\times C_p$. By Corollary A, $\BQ G$ is $*$-clean, but that contradicts Theorem A.
\end{proof}

{\bfseries Acknowledgments}: This research was partially supported by FINAPESQ-UEFS, Brazil.

\end{document}